\documentclass{article}
\usepackage{amssymb}
\usepackage{amsmath}
\usepackage{amsthm}
\usepackage[english]{babel}
\usepackage[letterpaper,top=2cm,bottom=2cm,left=3cm,right=3cm,marginparwidth=1.75cm]{geometry}
\usepackage{amsmath}
\usepackage[colorlinks=true, allcolors=blue]{hyperref}
\usepackage{todonotes}

\usepackage{thmtools}
\usepackage{cleveref}
\usepackage{diagbox}
\usepackage{xspace}

\usepackage{caption}
\usepackage{subcaption}
\usepackage{float}
\usepackage{mleftright}
\usepackage{multicol}
\usepackage{multirow}
\usepackage{enumitem}
\usepackage{comment}
\usepackage{mathtools}

\newtheorem{theorem}{Theorem}[section]
\newtheorem{proposition}[theorem]{Proposition}
\newtheorem{corollary}[theorem]{Corollary}
\newtheorem{lemma}[theorem]{Lemma}

\newtheorem{definition}[theorem]{Definition}

\newcommand{\Exp}[1]{\mathbb{E}\mleft[#1\mright]}

\newcommand{\pr}[1]{\mathrm{Pr}\mleft(#1\mright)}

\newcommand{\cl}[1]{\mathrm{cl}\mleft(#1\mright)}

\newcommand{\stirling}[2]{S\mleft(#1,#2\mright)}

\newcommand{\mobius}{M{\"o}bius\xspace}

\newcommand{\eqdef}{\coloneqq}

\newlist{propenum}{enumerate}{1}
\setlist[propenum]{label=\arabic*., ref=\theproposition~(\arabic*)}
\crefalias{propenumi}{proposition}

\title{Time to Cycle}
\author{Nir Lavee\thanks{School of Computer Science and Engineering, Hebrew University, Jerusalem 91904, Israel. e-mail: nir.lavee@mail.huji.ac.il.} \and  {Nati Linial\thanks{School of Computer Science and Engineering, Hebrew University, Jerusalem 91904, Israel. e-mail: nati@cs.huji.ac.il.{~Supported in part by an ERC Grant 101141253, "Packing in Discrete Domains - Geometry and Analysis".}}}}
\date{}
\begin{document}
\maketitle

\begin{abstract}
Consider the random process that starts with $n$ vertices and no edges,
where the edges of $K_n$ are
added one at a time in a uniformly chosen random order 
$e_1, e_2,\ldots, e_{\binom{n}{2}}$. Let
$T$ be the earliest time at which $e_1$ belongs to a cycle in this evolving random graph.
By solving the appropriate graph enumeration problem we show that $\mathbb{E}[T]=n$.
This fact turns out to be an instance
of a much more general phenomenon and we are able to
extend this theorem to all graphs and even to every matroid.
\end{abstract}

\section{Introduction}
The evolution of random graphs has been extensively studied since its introduction 
by Erdős and Rényi \cite{erdHos1960evolution}. We know when the giant component emerges,
when the graph becomes connected, when particular subgraphs appear, and much more. 
Cycles and bridges in the emerging graph have also received much attention, as 
reviewed e.g., in 
\cite{Boll2001,janson2011random,frieze2015introduction}.

We have stumbled numerically upon the following intriguing phenomenon.
Start with an edgeless graph on $n$ vertices and add edges one by one, chosen
uniformly at random,
$e_1, e_2,\ldots, e_{\binom{n}{2}}$. Stop at the earliest
step at which $e_1$ ceases being a bridge and becomes part of a cycle. The expectation
of this stopping time appeared to be exactly $n$, as we verified numerically for 
all $n\le 20$. 
Standard facts from random graph theory easily yield that this stopping time is $\Theta(n)$,
but we found it intriguing that the expected stopping time turned out to be exactly $n$. 
As shown here, a similar result applies, in fact, much more broadly and
holds in general for all matroids.

We present a proof for the case of graphs in \Cref{sec:graphs}. 
The general proof for matroids appears in \Cref{sec:matroids}.

\section{Preliminaries}
\label{sec:prelims}
Recall the Stirling numbers of the second kind:
$\stirling{n}{k}$ denotes the number of maps from $[n]$ onto $[k]$ \cite{stanley2011enumerative}.
Let us collect some basic properties of these numbers.
\begin{proposition}
We have:
\begin{propenum}
\item\label{prop:stirling_k_0}
For every $n\ge k\ge 1$, there holds
$$\stirling{n}{k}= \stirling{n-1}{k}\cdot k+\stirling{n-1}{k-1}.$$    
\item\label{prop:stirling_k_1}
For every $n\ge 2$,
$$
\sum_{k=1}^{n} (-1)^{k-1} \cdot (k-1)! \cdot \stirling{n}{k}=0.
$$
\item\label{prop:stirling_k_2} 
For every $n\ge 1$,
$$
\sum_{k= 2}^n (-1)^{k} \cdot (k-2)! \cdot \stirling{n}{k} = n-1.
$$
\end{propenum}
\end{proposition}

\begin{proof}
The first claim follows by counting separately those onto maps $f:[n]\to[k]$ for
which there exists some $x\neq 1$ in $[n]$ with $f(x)=f(1)$
resp.\ those for which no such $x$ exists.\\
To prove the second claim we apply the above identity
to this expression and cancel equal terms.\\
The third claim is easily verified for $n=1$ and $n=2$. 
Proceeding by induction and using Item \ref{prop:stirling_k_0}, the sum expands to
\[
\sum_{k=2}^n (-1)^k \cdot (k-2)! \cdot \stirling{n-1}{k}\cdot k+\sum_{k=2}^n (-1)^k \cdot (k-2)! \cdot \stirling{n-1}{k-1}.
\]
The latter term vanishes by Item \ref{prop:stirling_k_1}. We rewrite the first sum as
\[
\sum_{k=2}^n (-1)^k \cdot (k-2)! \cdot \stirling{n-1}{k}\cdot (k-1) + \sum_{k=2}^n (-1)^k \cdot (k-2)! \cdot \stirling{n-1}{k}. 
\]
The second sum is $n-2$ by induction, and by Item \ref{prop:stirling_k_1} 
the first term equals $1$.
\end{proof}

\begin{proposition}
\label{prop:binom_div}
    For every $n\ge m \ge 0$, we have
    \[
    \sum_{t=0}^m \frac{\binom{m}{t}}{\binom{n}{t}} = \frac{n+1}{n+1-m}.
    \]
\end{proposition}
\begin{proof}
    By a well-known variant of the hockey-stick identity (see e.g., \cite{graham94} chapter 5.1), we have
    \[\sum_{s=0}^m \binom{n-m+s}{s}=\binom{n+1}{m}\]
    and therefore,
    \[
    \sum_{t=0}^m \frac{\binom{m}{t}}{\binom{n}{t}}=\frac{m!}{n!}\sum_{t=0}^m \frac{(n-t)!}{(m-t)!}=\frac{m!}{n!}\sum_{s=0}^m \frac{(n-m+s)!}{s!}=\frac{m!\cdot (n-m)!}{n!}\cdot \binom{n+1}{m}=\frac{n+1}{n+1-m}.
    \]
\end{proof}

\subsection{Matroids}
We use standard matroid terminology, e.g., as in \cite{oxley2006matroid, white1987combinatorial, zaslavsky1975facing}. 
In what follows, $M=(E,I)$ is a matroid with rank function $r$, and 
$\mathcal{F}_M=\mathcal{F}$ is the collection of $M$'s flats. The set $\mathcal{F}$
carries the structure of a geometric lattice and
its \mobius function is denoted by $\mu$. Recall
that if $M$ is loopless, then its {\em characteristic polynomial} is defined as
\begin{equation}\label{eq:mu}
p_M(\lambda)=\sum_{A\in\mathcal{F}} \mu(\varnothing,A) \cdot \lambda^{r(E)-r(A)}.
\end{equation}
Let us recall that for a flat $A\in \mathcal{F}$, the contracted matroid $M/A$ is loopless. Applying \Cref{eq:mu} to $M/A$, taking
the derivative and evaluating at $\lambda=1$, we have
\[
p'_{M/A}(1)=\sum_{B\in\mathcal{F}_{M/A}} \mu_{M/A}(\varnothing,B)\cdot (r_{M/A}(E_{M/A})-r_{M/A}(B)).
\]
But $p'_{M/A}(1)$ is also 
expressible in terms of the rank and $\mu$ functions of $M$. Every flat $B\in\mathcal{F}_{M/A}$ corresponds to a flat $B'=A\cup B\in\mathcal{F}$ such that $r(B')=r_{M/A}(B)+r(A)$ and $\mu_{M/A}(\varnothing,B)=\mu(A,B')$. Therefore
\[
p'_{M/A}(1)=\sum_{B\in\mathcal{F}} \mu(A,B)\cdot (r(E)-r(B)).
\]
We remark that $(-1)^{r(E)-1} \cdot p'_M(1)$ is known as the beta invariant 
of $M$ \cite{crapo1967higher}. We also recall \mobius inversion: For any two functions $f,g:\mathcal{F}\to\mathbb{R}$, we have
\[
f(H)=\sum_{\substack{A\in\mathcal{F}\\ \text{~s.t.~} H\subseteq A}} g(A)\quad\quad\text{if and only if}\quad\quad g(H)=\sum_{\substack{B\in\mathcal{F}
}
} \mu(H,B)\cdot f(B).
\]
In particular, for any $f:\mathcal{F}\to\mathbb{R}$ we can define $g$ by the latter expression, and obtain
\begin{equation}
    \label{eq:mu_identity}
    f(H)=\sum_{\substack{A\in\mathcal{F}\\ \text{~s.t.~} H\subseteq A}} g(A)=
    \sum_{\substack{A,B\in\mathcal{F}\\ \text{~s.t.~} H\subseteq A}} \mu(A,B) \cdot f(B).
\end{equation}
We require the following immediate consequences of \mobius inversion.
\begin{proposition}
\label{prop:mu_intermediate_sum}
    If $M$ is loopless,
    \[
    \sum_{A,B\in\mathcal{F}} \mu(A,B)\cdot (|E|-|A|) \cdot(r(E)-r(B))=|E|.
    \]
\end{proposition}
\begin{proof}
    Expand the left-hand-side to obtain three sums:
    \[
    \sum_{A,B\in\mathcal{F}} \mu(A,B)\cdot (|E|-|A|) \cdot r(E)-\sum_{A,B\in\mathcal{F}} \mu(A,B)\cdot |E| \cdot r(B)+\sum_{A,B\in\mathcal{F}} \mu(A,B)\cdot |A|\cdot r(B) 
    \]
    The first sum vanishes: Note that $A=E$ does not contribute to the sum, and for $A\subsetneq E$ the sum $\sum_{B\in\mathcal{F}}\mu(A,B)$ is $0$, because by definition,
    \[\mu(A,E)=-\sum_{\substack{B\in\mathcal{F}\\\text{~s.t.~} B\subsetneq E}} \mu(A,B).\]
    The second sum also vanishes, by \mobius inversion. Let $f(B)\eqdef r(B)$ and $H=\varnothing$ in \Cref{eq:mu_identity}:
    \[
    \sum_{A,B\in\mathcal{F}} \mu(A,B)\cdot r(B)=r(\varnothing)=0.
    \]
    It remains to show that the third sum indeed satisfies
    \[
    \sum_{A,B\in\mathcal{F}} \mu(A,B)\cdot |A|\cdot r(B)=|E|.
    \]
    Rearranging the left-hand-side,
    \[
    \sum_{A,B\in\mathcal{F}} \mu(A,B)\cdot |A|\cdot r(B)=\sum_{A,B\in\mathcal{F}} \mu(A,B)\cdot r(B) \cdot \sum_{e\in A}1=\sum_{e\in E}\sum_{\substack{A,B\in\mathcal{F}\\\text{~s.t.~} e\in A}} \mu(A,B)\cdot r(B).
    \]
    It is sufficient to show that the inner sum is $1$, so that the total is $|E|$. Again we appeal to \mobius inversion. Let $f(B)\eqdef r(B)$ and $H=\cl{\{e\}}$ in \Cref{eq:mu_identity}:
    \[
    \sum_{\substack{A,B\in\mathcal{F}\\\text{~s.t.~}e\in A}} \mu(A,B)\cdot r(B)=r(\cl{\{e\}})
    \]
    $M$ is loopless, so the rank of a single element is $1$.
\end{proof}
\begin{proposition}
\label{prop:mu_contr_sum}
    For every flat $H\in\mathcal{F}$,
    \[
     \sum_{\substack{A\in\mathcal{F}\\\text{~s.t.~} H\subseteq A}} (|E|-|A|) \cdot p'_{M/A}(1)=|E|-|H|.
    \]
\end{proposition}
\begin{proof}
    Expanding the characteristic polynomial, the left-hand-side is
    \[
     \sum_{\substack{A,B\in\mathcal{F}\\\text{~s.t.~} H\subseteq A}} \mu(A,B)\cdot (|E|-|A|)  \cdot (r(E)-r(B)).
    \]
    Consider the same sum expressed in terms of the contraction $M/H$:
    \[
     \sum_{A,B\in\mathcal{F}_{M/H}} \mu_{M/H}(A,B)\cdot (|E_{M/H}|-|A|)  \cdot (r_{M/H}(E_{M/H})-r_{M/H}(B)).
    \]
    Applying \Cref{prop:mu_intermediate_sum}, this indeed equals $|E_{M/H}|=|E|-|H|$.
\end{proof}
\begin{proposition}
[e.g., chapter 7 in \cite{white1987combinatorial}]
    \label{prop:mu_prank}
    We have
    \[
    \sum_{A\in\mathcal{F}} p'_{M/A}(1)=r(E).
    \]
\end{proposition}
\begin{proof}
Expand the characteristic polynomial:
\begin{align*}
\sum_{A,B\in\mathcal{F}}
\mu(A,B) \cdot  (r(E)-r(B))
\end{align*}
Now use \mobius inversion on $f(B)\eqdef r(E)-r(B)$. Substituting $H=\cl{\varnothing}$ in \Cref{eq:mu_identity}, we have
\[\sum_{A,B\in \mathcal{F}} \mu(A,B) \cdot (r(E)-r(B))=f(\cl{\varnothing})=r(E).
\]
\end{proof}

\section{Graphs}
\label{sec:graphs}
Let $G=(V,E)$ be an order-$n$ graph and let $e\in E$ be an edge.
We say that the edge $e$ is {\em split} in a subgraph $H$ of $G$
if its two vertices belong to different connected components of $H$.
Consider the random process that starts with $n$ vertices and no edges. The
edges in $E\setminus \{e\}$ are inserted one at a time in a uniformly
chosen random order. If $e$ is not a bridge in $G$, we denote by $T_e$ the earliest
step at which $e$ is not split. 
If $e$ is a bridge, we let $T_e\eqdef |E|$.
We can express $\Exp{T_e}$ as
\[\Exp{T_e}=\sum_{t=1}^{|E|}t\cdot \pr{T_e=t}=\sum_{t=0}^{|E|-1} \pr{T_e> t}.
\]
The event $T_e> t$ occurs when $e$ is split in
the subgraph formed by the first $t$ selected edges. Let 
\[d(e,t)\eqdef| \{t\text{-edge subgraphs of~}G\setminus e \text{~in which~}e\text{~is split}\}|\] 
Fix a subgraph $H$ where $e$ is split. Of the
$(|E|-1)!$ permutations of $E\setminus \{e\}$, 
there are $t!\cdot (|E|-1-t)!$ whose first $t$ edges coincide with $E(H)$. Therefore,
\begin{equation}
\label{eq:exp_Te}
\Exp{T_e}=\sum_{t=0}^{|E|-1} \frac{d(e,t)\cdot t!\cdot (|E|-1-t)!}{(|E|-1)!}=\sum_{t=0}^{|E|-1} \frac{d(e,t)}{\binom{|E|-1}{t}}.
\end{equation}

\begin{definition}
Let $G=(V,E)$ be a graph, and let  $k\ge 1$ be an integer. 
We denote by $\Pi_k(V)$ the collection of all
partitions of $V$ into $k$ non-empty subsets. Given a partition $A\in \Pi_k(V)$, let 
\[
E_A\eqdef\{\{x,y\}\in E : \text{$x$ and $y$ are in the same part in $A$}\}.
\]
Edges in $E_A$ resp.\ $E\setminus E_A$ are said to be {\em internal}
resp.\ {\em external} to $A$.
\end{definition}

\begin{lemma}
\label{lemma:graph_esum}
For every graph $G=(V,E)$, an edge $e\in E$ and $0\le t\le |E|-1$, there holds
\begin{equation}
\label{eq:_graph_desum}
\sum_{e\in E} d(e,t)=\sum_{k\ge 2} (-1)^{k} \cdot (k-2)!  \sum_{A\in \Pi_k(V)} \binom{|E_A|}{t}\cdot (|E|-|E_A|).
\end{equation}
\end{lemma}
\begin{proof}
Consider the rightmost sum with some fixed $k\ge 2$.
It counts ways to choose a $k$-partition of $V$, 
say $A\in \Pi_k(V)$, a set $R\subseteq E_A$
of $A$-internal edges with $|R|=t$ and an $A$-external edge $e\notin E_A$. Given
a set $R$ of $t$ edges in $G$ and an edge $e\notin R$, we say that a
$k$-partition $A\in \Pi_k(V)$ is $(R,e)$-{\em compatible} if $R\subseteq E_A$
and $e\notin E_A$. In other words, given $R$ and $e$, 
this sum counts the $(R,e)$-compatible $k$-partitions 
$A\in \Pi_k(V)$. If $C_1,\ldots,C_j$ are the
connected components of the graph $(V,R)$, then necessarily
\begin{equation}\label{eq:refine}
\text{The partition~}\{C_1,\ldots,C_j\}\text{~of~} V 
\text{~is a refinement of the partition~} A.
\end{equation}
Also, since $e\notin E_A$, this edge must 
connect between two  distinct component of $(V,R)$, say $C_1$ and $C_2$. 
Consequently, $C_1$ and $C_2$ must be included in different parts of the partition $A$. 
To count $(R,e)$-compatible $k$-partitions $A$, 
we need to specify where the different $C_i$ go.
In view of (\ref{eq:refine}) and the previous comment we have to collect
$m$ of the $C_i$ including $C_1$ and excluding $C_2$, 
to be contained in some part of $A$,
and then partition the other $C_j$ into the remaining $k-1$ parts of $A$. This yields:
\[
\text{The number of~}(R,e)\text{-compatible~}k\text{-partitions~}A\text{~is~}
\sum_{m=1}^{j-1} \binom{j-2}{m-1} \cdot \stirling{j-m}{k-1}.
\]
Therefore, the contribution of $(R,e)$-compatible $k$-partitions to the right-hand-side of \Cref{eq:_graph_desum} is
\[
\sum_{k\ge 2} (-1)^k \cdot (k-2)!\sum_{m=1}^{j-1} \binom{j-2}{m-1} \cdot \stirling{j-m}{k-1},
\]
which we turn to evaluate. We rearrange this sum as:
\[
\sum_{m=1}^{j-1} \binom{j-2}{m-1}  \sum_{k\ge 2} (-1)^k \cdot (k-2)!\cdot \stirling{j-m}{k-1}.
\]
By \Cref{prop:stirling_k_1}, the inner sum vanishes if $j-m\ge 2$. For $m=j-1$, the 
inner sum is $1$, and also $\binom{j-2}{m-1}=1$. 
Therefore, every choice of
$R$ and $e$ contributes $1$ to the sum total. It follows that the
right-hand-side of \Cref{eq:_graph_desum} counts pairs $(R,e)$
where $R\subset E$ is a set of $t$ edges, and $e\in E$ 
connects two distinct components of $(V,R)$.
This equals the left-hand-side of \Cref{eq:_graph_desum}, i.e., $\sum_{e\in E} d(e,t)$.
    
\end{proof}

Let us denote by $\Pi'_k(V)\subseteq \Pi_k(V)$ the collection of those 
$k$-partitions of $V$ for which not all of $G$'s edges are internal to $A$.

\begin{theorem}
\label{thm:graph_det_sum}
If a graph $G=(V,E)$ has $n$ vertices and $c<n$ connected components, then
\[
\frac{1}{|E|}\sum_{e\in E} \sum_{t=0}^{|E|-1} \frac{d(e,t)}{\binom{|E|-1}{t}} = n-c.
\]
\end{theorem}

\begin{proof}
Applying \Cref{lemma:graph_esum} to the left-hand-side, we have
\begin{equation}
\label{eq:graph_det_sum}
\frac{1}{|E|}\sum_{e\in E} \sum_{t=0}^{|E|-1} \frac{d(e,t)}{\binom{|E|-1}{t}} =\frac{1}{|E|}\sum_{t=0}^{|E|-1} \frac{1}{\binom{|E|-1}{t}} \sum_{k\ge 2} (-1)^{k} \cdot (k-2)!  \sum_{A\in \Pi_k(V)} \binom{|E_A|}{t}\cdot (|E|-|E_A|).
\end{equation}
We may restrict the sum in \Cref{eq:graph_det_sum}
to $A\in\Pi'_k(V)$ so that $|E_A|<|E|$, since the other terms of the sum
vanish. Observe that only the binomial coefficients in \Cref{eq:graph_det_sum} 
depend on $t$. By \Cref{prop:binom_div}, we have
\[
\sum_{t=0}^{|E|-1} \frac{\binom{|E_A|}{t}}{\binom{|E|-1}{t}}=\frac{|E|}{|E|-|E_A|}.
\]
Applying this to \Cref{eq:graph_det_sum}, the term cancels and leaves $1$ for every $A\in\Pi'_k(V)$:
    \[
     \frac{1}{|E|}\sum_{e\in E} \sum_{t=0}^{|E|-1} \frac{d(e,t)}{\binom{|E|-1}{t}}=\sum_{k\ge 2} (-1)^k \cdot (k-2)! \cdot |\Pi'_k(V)|
    \]
There are $\stirling{n}{k}$ ways to partition $V$ into $k$ non-empty subsets. 
Of those, the partitions in $\Pi_k'(V)$ exclude the 
$\stirling{c}{k}$ partitions in which every connected component is contained in one of the parts. Hence $|\Pi'_k(V)|=\stirling{n}{k}-\stirling{c}{k}$, and the claim follows by \Cref{prop:stirling_k_2}.
\end{proof}

\begin{corollary}
Let $G=(V,E)$ be a graph with $n$ vertices, $c$ connected components, and no bridges. Adding the edges one by one in a uniformly random order, the expected number of steps until the first edge resides in a cycle is $n-c+1$.
\end{corollary}
\begin{proof}
Given that every edge eventually joins a cycle, the expected time is the average of \Cref{eq:exp_Te} over edges as calculated in \Cref{thm:graph_det_sum}, with the addition of $1$ to account for the initial edge.
\end{proof}

\paragraph{Counting graphs with a disconnected pair.}
Let $d(n,t)$ be the number of graphs with $n$ vertices and $t$ edges such that two particular vertices, say $1$ and $2$, reside in distinct connected components. 
Complementing this, $c(n,t)=\binom{\binom{n}{2}}{t}-d(n,t)$ counts
graphs where $1$ and $2$ are in the same component. 
Several small values of these functions are shown in \Cref{fig:dnt_cnt}.
% Side-by-side figures
\begin{figure}[H]
\centering
\begin{subfigure}[t]{.4\textwidth}
\centering
\begin{tabular}{cc|cccc}
 & \multicolumn{1}{c}{} & \multicolumn{4}{c}{$t$}\tabularnewline
 & \multicolumn{1}{c}{} & $0$ & $1$ & $2$ & $3$\tabularnewline
\cline{3-6} \cline{4-6} \cline{5-6} \cline{6-6} 
\multirow{3}{*}{$n$} & $2$ & $1$ &  &  & \tabularnewline
 & $3$ & $1$ & $2$ &  & \tabularnewline
 & $4$ & $1$ & $5$ & $8$ & $2$\tabularnewline
\end{tabular}
\end{subfigure}%
\begin{subfigure}[t]{.4\textwidth}
\centering
\begin{tabular}{cc|cccccc}
 & \multicolumn{1}{c}{} & \multicolumn{6}{c}{$t$} \tabularnewline
 & \multicolumn{1}{c}{} & $1$ & $2$ & $3$ & $4$ & $5$ & $6$\tabularnewline
\cline{3-8} \cline{4-8} \cline{5-8} \cline{6-8} \cline{7-8} \cline{8-8} 
\multirow{3}{*}{$n$} & $2$ & $1$ &  &  &  &  & \tabularnewline
 & $3$ & $1$ & $3$ & $1$ &  &  & \tabularnewline
 & $4$ & $1$ & $7$ & $18$ & $15$ & $6$ & $1$\tabularnewline
\end{tabular}
\end{subfigure}
\caption{Values of $d(n,t)$ (left) and $c(n,t)$ (right). Row sums of $c(n,t)$ are \cite[A224510]{oeis}.}
\label{fig:dnt_cnt}
\end{figure}
We arrive at the following identity:
\begin{corollary}
\label{cor:graph_identity}
    For every $n\ge 2$,
    \[
    \sum_{t=0}^{\binom{n}{2}-1} \frac{d(n,t)}{\binom{\binom{n}{2}-1}{t}}=n-1.
    \]
\end{corollary}
\begin{proof}
    Apply \Cref{thm:graph_det_sum} to $K_n$.
\end{proof}
It is possible to compute $d(n,t)$ as follows. Let $q(n,t)$ be the number of connected
graphs on $n$ vertices with $t$ edges. Say that vertex $1$ is in a connected component 
with $k$ vertices and $s$ edges, and choose
any graph with $t-s$ edges on the remaining $n-k$ vertices. Then:
\[
d(n,t)=\sum_{k=1}^{n-1} \binom{n-2}{k-1} \sum_{s=0}^{t} q(k,s) \cdot \binom{\binom{n-k}{2}}{t-s}
\]
The quantity $q(k,s)$ may be computed using exponential generating functions, see
e.g., \cite{flajolet2009analytic,harary2014graphical}. One can also consider the exponential generating function of $d$. Let:
\[
\mathcal{D}(x,y)=\sum_{n,t} \frac{x^n}{n!} \cdot y^t \cdot d(n,t),\quad\mathcal{Q}(x,y)=\sum_{n,t} \frac{x^n}{n!} \cdot y^t \cdot q(n,t),\quad\mathcal{G}(x,y)=\sum_{n,t} \frac{x^n}{n!} \cdot y^t \cdot \binom{\binom{n}{2}}{t}.
\]
A standard argument in the theory of exponential generating functions yields 
$\mathcal{D}''=\mathcal{Q}' \cdot \mathcal{G}'$ where the derivatives are with respect to $x$. We initially attempted to prove \Cref{cor:graph_identity} directly from such equalities, and eventually turned to the methods described in \Cref{lemma:graph_esum} and \Cref{thm:graph_det_sum}.

\section{Matroids}
\label{sec:matroids}
Let $M=(E,I)$ be a matroid. Pick some element $e\in E$, and
consider the random process where the elements of $E\setminus\{e\}$ are
collected one at a time in a uniformly random order.
Define the random variable $T_e$ as follows:
If $e$ is contained in a circuit, then $T_e$ is
the earliest step at which $e$ is in the closure of previously chosen
elements. If $e$ is an isthmus, we let $T_e\eqdef |E|$. Let $d(e,t)$ be the number of subsets $R\in\binom{E}{t}$ such that $e\notin\cl{R}$. 
The expression for $\Exp{T_e}$ in \Cref{eq:exp_Te} holds as well for the more general framework of matroids. 
We now develop the matroid-theoretic counterpart of
\Cref{lemma:graph_esum}.
\begin{lemma}
\label{lemma:matroid_esum}
Let $M=(E,I)$ be a matroid and let $0\le t\le |E|-1$. Then
\[
\sum_{e\in E} d(e,t)=
\sum_{A\in\mathcal{F}} \binom{|A|}{t} \cdot (|E|-|A|)\cdot p'_{M/A}(1).
\]
\end{lemma}

\begin{proof}
Every flat $A\in\mathcal{F}$ and subset $R\subseteq A$ of $t$ elements contribute $(|E|-|A|)\cdot p'_{M/A}(1)$ to the right-hand-side. We rearrange the sum by $R$: For every subset $R\subseteq E$ of $t$ elements, iterate over every flat $A$ that contains $R$. We note that $A$ contains $R$ if and only if it contains the flat $\cl{R}$. We can write the sum as
\[
\sum_{R\in\binom{E}{t}} \sum_{\substack{A\in\mathcal{F}\\\text{~s.t.~} \cl{R}\subseteq A}} (|E|-|A|) \cdot p'_{M/A}(1).
\]
Applying \Cref{prop:mu_contr_sum} with $H=\cl{R}$ to the inner sum, this simplifies to
\[
\sum_{R\in\binom{E}{t}} (|E|-|\cl{R}|)
\]
which counts the number of ways to choose $t$ elements from $E$ and another element not in their closure, and hence equals $\sum_{e\in E} d(e,t)$.
\end{proof}

\begin{theorem}
\label{thm:mat_det_sum}
The following holds for every matroid $M=(E,I)$:
\[
\frac{1}{|E|} \sum_{e\in E} \sum_{t=0}^{|E|-1} \frac{d(e,t)}{\binom{|E|-1}{t}}=r(E).
\]
\end{theorem}

\begin{proof}
By applying \Cref{lemma:matroid_esum} to the sum on the left-hand-side, we have
\begin{equation}
\label{eq:mat_det_sum}
\frac{1}{|E|} \sum_{e\in E} \sum_{t=0}^{|E|-1} \frac{d(e,t)}{\binom{|E|-1}{t}}=\frac{1}{|E|} \sum_{t=0}^{|E|-1} \frac{1}{\binom{|E|-1}{t}}\sum_{A\in \mathcal{F}} \binom{|A|}{t} \cdot (|E|-|A|) \cdot p'_{M/A}(1).
\end{equation}
We may restrict the sum to flats $A\in\mathcal{F}$ such that $A\subsetneq E$, since the other terms of the sum vanish. Observe that only the binomial coefficients in \Cref{eq:mat_det_sum} depend on $t$, and for every flat $A\in\mathcal{F}$ such that $A\subsetneq E$, we have by \Cref{prop:binom_div},
\[
\sum_{t=0}^{|E|-1} \frac{\binom{|A|}{t}}{\binom{|E|-1}{t}}=\frac{|E|}{|E|-|A|}.
\]
Applying this to \Cref{eq:mat_det_sum}, we have
\[
\frac{1}{|E|} \sum_{e\in E} \sum_{t=0}^{|E|-1} \frac{d(e,t)}{\binom{|E|-1}{t}}=\sum_{\substack{A\in\mathcal{F}\\\text{~s.t.~} A\subsetneq E}} p'_{M/A}(1).\]
We note that the restriction $A\subsetneq E$ does not affect the right-hand-side, as $p'_{M/E}(1)=0$. By \Cref{prop:mu_prank}, the sum of $p'_{M/A}(1)$ over all flats $A\in\mathcal{F}$ equals $r(E)$.
\end{proof}
\begin{corollary}
Let $M=(E,I)$ be a matroid such that every $e\in E$ is contained in a circuit. Adding the elements one at a time in a uniformly chosen random order, the expected number of steps until the first element joins a circuit is $r(E)+1$.
\end{corollary}
\begin{proof}
    Given that every $e\in E$ eventually joins a circuit, the expected value is the average of \Cref{eq:exp_Te} over elements as calculated in \Cref{thm:mat_det_sum}, with the addition of $1$ to account for the initial element.
\end{proof}

\section{Acknowledgments}
We thank Yuval Peled and Doron Zeilberger for their insightful comments.

\bibliographystyle{alpha}
\bibliography{main}

\end{document}